\documentclass[12pt]{amsart}

\usepackage{amsaddr}
\usepackage{amssymb, color}
\usepackage{dsfont}
\usepackage{enumerate}
\usepackage[normalem]{ulem}
\usepackage{xcolor}
\usepackage{ulem}

\newtheorem{lemma}[equation]
{Lemma}
\newtheorem{thm}[equation]
{Theorem}
\newtheorem{prop}[equation]
{Proposition}
\newtheorem{cor}[equation]
{Corollary}

\newcommand{\moblue}[1]{{#1}}

\title{ The solution of the Loewy-Radwan conjecture }
\author[M. Omladi\v{c}]{Matja\v{z} Omladi\v{c}}
\address{Faculty of Mathematics and Physics, University of Ljubljana, Slovenia}
\email{matjaz@omladic.net}

\author{ Klemen \v Sivic}
\address{Faculty of Mathematics and Physics, University of Ljubljana, 
Slovenia}
\email{klemen.sivic@fmf.uni-lj.si}

\begin{document}

\begin{abstract} A seminal result of Gerstenhaber gives the maximal dimension of a linear space of nilpotent matrices. It also exhibits the structure of such a space when the maximal dimension is attained. Extensions of this result in the direction of linear spaces of matrices with a bounded number of eigenvalues have been studied. In this paper, we answer what is perhaps the most general problem of the kind as proposed by Loewy and Radwan, by solving their conjecture in the positive. We give \moblue{the maximal dimension of a} vector space of $n\times n$ matrices with no more than $k<n$ eigenvalues. We also exhibit the structure of the spaces for which this dimension is attained.

\end{abstract}

\thanks{MO\&K\v{S} acknowledge financial support from the Slovenian Research Agency (research core funding No. P1-0222). K\v{S} also acknowledges grants N1-0103 and J1-3004 from the Slovenian Research Agency.}
\maketitle

\section{ Introduction }

This paper presents the positive solution of the Loewy-Radwan conjecture, which has been open for more than twenty years (Theorem \ref{thm:main}). It belongs to a theory that started over 60 years ago by the famous Gerstenhaber \cite{Ger} result on linear spaces of nilpotent matrices of maximal dimension. An interested reader may also wish to consider some recent results in the area, such as \cite{KBO} by Kokol Bukov\v{s}ek and Omladi\v{c}, and \cite{dSPI,dSPII,dSPIII} by de Seguins Pazzis. It seems that spaces of matrices satisfying more general conditions on eigenvalues have been studied for the first time by Omladi\v{c} and \v{S}emrl in \cite{OS}. We will first give a brief history of this theme.

Throughout the paper we fix positive integers $n$ and $k<n$ and suppose $V$ is a linear subspace of the space $M_n(\mathbb{C})$ of $n\times n$ complex matrices with the property that each member of $V$ has at most $k$ distinct eigenvalues. Here $\mathbb{C}$ can be replaced by any algebraically closed field of characteristic zero.  We are interested in how large the dimension of such a space can be. The case when all the matrices are assumed nilpotent dates back to Gerstenhaber \cite{Ger}, who proved that the dimension of such a space is at most $\displaystyle{n\choose 2}$. Actually, Gerstenhaber proved the result for all fields with at least $n$ elements, and this assumption was later removed (cf.\ Sere\v{z}kin \cite{Ser}, and Mathes, Omladi\v{c}, Radjavi \cite{MOR}).
Moreover, Gerstenhaber showed that when the maximal dimension is attained the space is simultaneously similar to the space of all strictly upper triangular matrices. Consequently, any space of matrices with only one eigenvalue has dimension at most $\displaystyle{n\choose 2}+1$, and in case of equality such a space is simultaneously similar to the space of upper triangular matrices with equal diagonal entries, see \cite{OS} and also \cite{dSP1} for \moblue{a} generalization of these results to other fields. The article \cite{OS} also contains the maximal possible dimension for a vector space of matrices with at most $k$ distinct eigenvalues when $k=2$ and $n$ is odd and when $k=n-1$ under some additional assumptions. In these two cases the spaces of maximal dimension were also classified in \cite{OS}.

Later, Loewy and Radwan \cite{LR} removed the assumptions needed in \cite{OS} and showed that the dimension of a vector space of matrices with most 2, respectively $(n-1)$, distinct eigenvalues is at most $\displaystyle{n\choose 2}+2$, respectively $\displaystyle{n\choose 2}+{n-1\choose 2}+1$. They also showed that for $k=3$ the corresponding upper bound for the dimension is $\displaystyle{n\choose 2}+4$ and conjectured that the upper bound is $\displaystyle{n\choose 2}+{k\choose 2}+1$ for every $k<n$. On the other hand, de Seguins Pazzis \cite{dSP} classified $\displaystyle\left({n\choose 2}+2\right)$-dimensional spaces of matrices having at most 2 distinct eigenvalues and extended the results to other fields. We also note that Jordan algebras of matrices with few eigenvalues were studied in \cite{GKOR} and that Gerstenhaber's theorem was generalized to semisimple Lie algebras \cite{DKK,MR}, which was further used to give another proof of \moblue{the} Erd\H{o}s-Ko-Rado theorem in combinatorics \cite{Woo}.

The aim of this paper is to prove the Loewy-Radwan conjecture and to classify spaces of matrices with at most $k$ distinct eigenvalues and which have the maximal possible dimension among such spaces.  More precisely, we are going to show the following.

\begin{thm}\label{thm:main}
Let $n$ and $k<n$ be positive integers and let $V$ be a linear subspace of $M_n(\mathbb{C})$ with the property that each member of $V$ has at most $k$ distinct eigenvalues. Then
$$\dim V\le {n\choose 2}+{k\choose 2}+1.$$
Moreover, if the equality holds and $k\ge 3$, then there exists $p\in\{0,1,\ldots ,n-k+1\}$ such that $V$ is simultaneously similar to the space of all matrices of the form
\begin{equation}\label{eq:required_form}
\begin{pmatrix}
A & B & C\\0 & D & E\\0 & 0 & F
\end{pmatrix}
\end{equation}
where $B\in M_{p\times (k-1)}(\mathbb{C})$, $D\in M_{k-1}(\mathbb{C})$ and $E\in M_{(k-1)\times (n-k-p+1)}(\mathbb{C})$ are arbitrary and
$\begin{pmatrix}
A & C\\0 & F
\end{pmatrix}$
is an arbitrary upper triangular matrix with equal diagonal entries.
\end{thm}

The proof of this theorem is inspired by the proof of the generalization of Gerstenhaber's theorem to semisimple Lie algebras given in \cite{DKK}. Although our proof is technically more challenging, it is possible to adapt some of the main ideas from \cite{DKK} to our situation. Let us briefly explain these main ideas
. To show that $\displaystyle\dim V\le {n\choose 2}+{k\choose 2}+1$, we first observe that $V$ belongs to some (projective) subvariety of a Grassmannian variety, which is invariant for the action by conjugation of the (solvable) group of invertible upper triangular matrices. The Borel Fixed Point Theorem then enables us to reduce the problem to 
linear subspaces that are invariant under conjugation by invertible upper triangular matrices.
Such subspaces are spanned by diagonal matrices and matrix units, therefore their dimensions and the number of eigenvalues of their members can be estimated in a straightforward way.

To classify the subspaces of maximal dimension, we use an induction on $k$. We first show that our space must contain a nonderogatory (or cyclic) matrix with $k-1$ simple eigenvalues. With no loss of generality we assume that this matrix is in the Jordan canonical form.
The rest of the proof is based on the following idea.
We define a group homomorphism ${\phi} \colon t\mapsto
\begin{pmatrix} t & 0 \\ 0 & I_{n-1}
\end{pmatrix}$ and consider the spaces $V_0=\lim_{t\to 0}{{}\phi} (t)V{{}\phi} (t)^{-1}$ and $V_{\infty}=\lim_{t\to 0}{{}\phi} (t)^{-1}V{{}\phi}(t)$.  These spaces are invariant for the action $(t,W)\mapsto \phi(t)W\phi(t)^{-1}$ of the group $\mathbb{C}^{\ast}=(\mathbb{C}\setminus\{0\},\cdot)$, they have the same dimension as $V$ and their members have at most $k$ distinct eigenvalues. We first consider the structure of such spaces. Using the structure of $\mathbb{C}^\ast$-modules 
 and the condition on the number of eigenvalues we can compute that the lower-right $(n-1)\times (n-1)$ corner of  such $W$ 
has dimension exactly ${n-1\choose 2}+{k-1\choose 2}+1$ and its members have at most $k-1$ eigenvalues.
For $k\ge 4$ we use the inductive assumption, while for $k=3$ we use the main result of \cite{dSP} together with the existence of a nonderogatory matrix with $k-1$ simple eigenvalues in $W$ 
 to obtain the structure of  its 
 lower-right corner. 
 After that we show that each such $W$ is 
 of an appropriate form.  We apply this result to $V_0$ and $V_{\infty}$, and finally we show that $V_0=V_{\infty}$, which implies that $V$ is equal to these two spaces and concludes the proof of the theorem.
The above argument uses some results from representation theory, but we make it accessible to the reader by defining $V_0$ and $V_{\infty}$ in an equivalent way and then using only methods from linear algebra.

Section \ref{sec2} consists of the proof of the first part
 of {Theorem \ref{thm:main}} 
and the rest of the paper is devoted to the second part. 
 After showing some preliminary results in Section \ref{sec3pre}, in Section \ref{sec3} we prove the second part of Theorem 1 under an additional assumption that $V=\phi(t)V\phi(t)^{-1}$ for all $t\ne 0$. The general case is proved in Section \ref{sec:max}.

\section{ The {upper bound on the dimension}}\label{sec2}


{In this section we prove the first part of Theorem \ref{thm:main}, i.e. we show that a linear subspace of $M_n(\mathbb{C})$ whose members have at most $k$ distinct eigenvalues has dimension at most $\displaystyle{n\choose 2}+{k\choose 2}+1$. We will prove this with the help of algebraic geometry, therefore w}e first show that certain conditions on matrices are open in the Zariski topology. We call an $n\times n$ matrix \textit{regular} {or \textit{cyclic}} or {\textit{nonderogatory} whenever its centralizer is $n$-dimensional. It is \moblue{well known} (see e.g. \cite[Section 3.2.4]{HJ}) that this condition is equivalent to the condition that the characteristic and minimal polynomial of the matrix coincide or that the Jordan canonical form of the matrix has only one Jordan block for each eigenvalue.

\begin{lemma}\label{open_conditions}
Let $A$ be an $n\times n$ complex matrix and $k$ be a nonnegative integer. The following conditions are open in the Zariski topology.
\begin{enumerate}[(a)]
\item
$A$ is regular.
\item
$A$ has more than $k$ distinct eigenvalues.
\item
$A$ has more than $k$ simple eigenvalues (i.e., they have algebraic multiplicity 1).
\end{enumerate}
\end{lemma}

\begin{proof}
\begin{enumerate}[(a)]
\item
It is \moblue{well known} that $C(A)$, the centralizer of $A$, has dimension at least $n$  (see e.g. \cite[Section 3.2.4]{HJ}), so we may modify the defining condition of regularity into $\dim C(A)\le n$, meaning that $\dim \mathrm{ker}\,\mathrm{ad}_A \le n$, or equivalently $\mathrm{rank}\,\mathrm{ad}_A \ge n^2-n$, where $\mathrm{ad}_A\,:\, M_n(\mathbb{C})\rightarrow M_n(\mathbb{C})$ is defined in the usual way by letting $\mathrm{ad}_A: B\mapsto [A,B]{=AB-BA}$. This amounts to the same as requiring that at least one of the $(n^2-n)$-minors of the matrix of the transformation $\mathrm{ad}_A$ in a fixed basis of $M_n(\mathbb{C})$ is different from zero -- clearly an open condition.
\item
Let $p_A$ be the characteristic polynomial of $A$. Then it is \moblue{well known} and not hard to see that condition (b) is fulfilled if and only if the degree of gcd$(p_A,p_A')$ is {smaller} than {$n-k$}. {It is also a classical result (see e.g. \cite[Section 2.1]{KN}) that} this degree equals $2n-1-${rank}$\,S_{p_A,p_A'}$, where $S_{p_A,p_A'}$ is the Sylvester matrix of the polynomials $p_A$ and $p_A'$. Let us recall that (in block partition made of $n-1$ respectively $n$ rows)
  \[
    S_{p_A,p_A'}=
    \begin{pmatrix}
      a_n & a_{n-1} & \cdots & \cdots & a_0 &   &   &   \\
        & a_n & a_{n-1} & \cdots & \cdots & a_0 &   &   \\
        &   & \ddots &  \ddots  &   &   & \ddots &   \\
        &   &   & a_n & a_{n-1} & \cdots & \cdots & a_0 \\
      b_{n-1} & \cdots &  b_0 &   &   &   &   &   \\
        & b_{n-1} &  \cdots & b_0 &   &   &   &   \\
        &   & \ddots &   & \ddots &   &   &   \\
	&   &   & \ddots &   & \ddots &   &   \\
        &   &   &   & \ddots &   &  \ddots &   \\
        &   &   &   &   & b_{n-1}  & \cdots & b_0
    \end{pmatrix}.
  \]
  Here, the $a_k$'s denote the coefficients of $p_A$ and the $b_j$'s the coefficients of its derivative. So, the matrix $A$ has more than $k$ distinct eigenvalues if and only if the $(n+k)\times(n+k)$ minors of the matrix $S_{p_A,p_A'}$ are not all zero\moblue{,} which is clearly an open condition.
\item
We will show that the condition that $p_A$ has at most $k$ simple roots is closed. This condition is equivalent to the condition that the matrix $p_A'(A)$ has at most $k$ nonzero eigenvalues (counted with algebraic multiplicities), which is further equivalent to $\mathrm{rank}\, p_A'(A)^n\le k$, a closed condition.
\end{enumerate}
\end{proof}

}



Here we follow some concepts used in \cite[Section~3]{ORS}, so that we will give only the general ideas and omit some of the details.
Fix an integer $k$, $1\le k\le n-1$ and define
\[
    X=\{{A\in M_n(\mathbb{C}); A\ \mbox{has no more than}\ k\ \mbox{distinct eigenvalues}}\}.
\]
We will denote by $m$ the maximal possible dimension of a vector space $V$ such that $V\subseteq X$. {By Lemma \ref{open_conditions}(b) the set $X$ is closed in the Zariski topology and it is clearly homogeneous (i.e. if $A\in X$ and $\alpha \in \mathbb{C}$, then $\alpha A\in X$), so we may view it as a projective variety.} Following \cite[Example 6.19]{Har} 
 we introduce the Fano variety
\[
    \begin{split}
       F_m(X)= & \{V\subseteq X;V\ \mbox{a vector space,}\ \dim V=m\}
    \end{split}
\]
which is a closed subset of the Grassmannian variety $\mathrm{Gr}({m,n^2})$ of all $m$-dimensional subspaces in $\mathbb{C}^{n^2}$, considered as a projective variety via the Pl\"ucker embedding that sends a vector space $V$ with a basis $\{v_1,v_2,\ldots ,v_m\}$ into $[v_1\wedge v_2\wedge \cdots \wedge v_m]\in \mathbb{P}\left(\wedge^m(\mathbb{C}^{n^2})\right).$ {The variety $F_m(X)$ is non-empty by the definition of the number $m$.}
Let {$\mathcal{T}_n$} be the solvable algebraic group of all {invertible} upper triangular matrices and define an action of {$\mathcal{T}_n$} on $\mathrm{Gr}({m,n^2})$ by $${\mathcal{T}_n\times \mathrm{Gr}(m,n^2)}\longrightarrow {\mathrm{Gr}(m,n^2)}, (P,V)\mapsto PVP^{-1}{=\{PAP^{-1};A\in V\}}.$$ It is obvious that $X$ is invariant under conjugation, and consequently $F_m(X)$ is invariant under the above action, which is given by regular (rational) maps:
$$(P,[A_1\wedge A_2\wedge \cdots\moblue{\wedge} A_m])\mapsto [PA_1P^{-1}\wedge PA_2P^{-1}\wedge \cdots \wedge PA_mP^{-1}].$$
So, we can apply the following theorem to $F_m(X)$:

\bigskip

\noindent\textsc{Borel fixed point theorem \cite[Theorem 10.4]{Bor}:} \textit{Let $Z$ be a {non-empty} projective variety and $G$ be a {connected} solvable algebraic group acting on it via regular maps. Then this action has a fixed point {in $Z$}.}

\bigskip

Using this result we conclude.

\begin{lemma}\label{lm:borel}
There exists a linear subspace $V\in F_m(X)$ such that $PAP^{-1} \in V$ for all $A\in V$ and every {invertible} upper triangular matrix $P$.
\end{lemma}

{{We now investigate properties of a space satisfying the previous lemma.} In the following lemma we denote the set $\{1,\ldots ,n\}$ by $[n]$. The proof of this lemma may be found in \cite[Lemma~9]{ORS}.} 
We present all of it for the sake of completeness, some of it will be used in Lemma \ref{lm:triang}, the rest of it may be of independent interest.

\begin{lemma}\label{lm:Eij} Let $V\subseteq M_n(\mathbb{C})$ be a vector space such that $PAP^{-1}\in V$ for all $A\in V$ and all invertible upper triangular $P\in M_n(\mathbb{C})$. Then:
  \begin{enumerate}[(a)]
    \item If for some $A\in V$ and $i,j\in[n]$ with $i<j$ we have $a_{ji}\ne0$, then $E_{ij}\in V$.
    \item If for some $i,j\in[n]$ with $i<j$ we have $E_{ij}\in V$, then $E_{pq}\in V$ for all ${p,q}\in[n]$ with ${p}\le i$ and ${q}\ge j$.
        {If we have $E_{ij}\in V$ for some $i,j\in[n]$ with $i>j$, then  the commutator $[E_{pq},E_{ij}]$ belongs to $V$ for all $p,q\in[n]$ with $p<q$.}
    \item If $V$ contains a matrix of the form $A=\begin{pmatrix}
                                                  \alpha & b^{T} \\
                                                  c & D
                                                \end{pmatrix}$ in the block partition determined by dimensions $(1,n-1)$, then it also contains \moblue{the} matrices $\begin{pmatrix}
                                                  0 & b^{T} \\
                                                  0 & 0
                                                \end{pmatrix}$ and $\begin{pmatrix}
                                                  0 & 0 \\
                                                  c & 0
                                                \end{pmatrix}$.
    \item Claim (c) remains valid if we replace the first column and row by the $i$-th column an row for any $i\in[n]$.
    \item If for some $A\in V$ and $i,j\in[n]$ with $i\ne j$ we have $a_{ij}\ne0$, then $E_{ij}\in V$.
    \item If for some $A\in V$ and $i,j\in[n]$ with $i< j$ we have $a_{ii}\ne a_{jj}$, then $E_{ij}\in V$.
  \end{enumerate}
\end{lemma}


{%
Using Lemma \ref{lm:Eij} we will {now} show that {a subspace of $M_n(\mathbb{C})$ which satisfies Lemma \ref{lm:borel} }
 has a very special form.

\begin{lemma}\label{lm:triang}
Let $V$ be a subspace of $M_n(\mathbb{C})$ which satisfies the conditions of Lemma \ref{lm:borel}.
Then  there exists a sequence of 
standard subspaces $U_{i_1},U_{i_2},$ $\ldots, U_{i_s}$ 
 of $\mathbb{C}^n$  of dimension at least 2 with trivial intersections such that 
 for each $t\in\{i_1,i_2,\ldots ,i_s\}$ the span
$$\emph{Span}\,\{e_i\,;\,\mbox{there exists}~e_j\in U_t\,\ 
\mbox{for some}~j\ge i\}$$ is invariant under all members of $V$, 
 and:
  \begin{enumerate}[(a)]
    \item For any index $t\in\{i_1,i_2,\ldots,i_s\}$ and any 
 standard basis vectors $e_i,e_j\in U_t$ we have $E_{ij}\in V$.
    \item All $E_{ij}$ belong to $V$ for {$1\le i<j\le n$.}
  \end{enumerate}
\end{lemma}

\begin{proof}
Recall that $V$ is of maximal possible dimension among spaces of matrices having at most $k$ distinct eigenvalues.
Let $W_1\le W_2\le\cdots\le W_r$ be a maximal chain of 
 subspaces of $\mathbb{C}^n$ which are invariant under all members of $V$ and spanned by some of the first vectors among $e_i$. 
If $\dim W_{t-1}<\dim W_t-1$, let $U_t$ be spanned by standard basis vectors $e_i\in W_t$ such that $e_i\not\in W_{t-1}$. Note that the inequality $\dim U_t\ge 2$ follows immediately.

{ Let $i\in \{1,2,\ldots ,n\}$ be arbitrary index such that $e_i\in U_t$ and $e_{i-1}\in U_t$ and suppose that $a_{ij}=0$ for all $A\in V$ and all indices $j<i$. Lemma \ref{lm:Eij}(e) and the second part of Lemma \ref{lm:Eij}(b) then imply that $a_{pq}=0$ for all $A\in V$ and all $p>i$ and $q<i$.
(Indeed, if $a_{pq}\ne0$ for some $q<i<p$, then $E_{pq}\in V$ by Lemma \ref{lm:Eij}(e) and then $E_{iq}=[E_{ip},E_{pq}]\in V$ by  Lemma \ref{lm:Eij}(b).)
However, then the linear span of $e_1,e_2,\ldots ,e_{i-1}$ is a $V$-invariant subspace of $\mathbb{C}^n$ which is strictly contained in $W_t$ (since it does not contain $e_i$) and strictly contains $W_{t-1}$ (since $e_{i-1}\not \in W_{t-1}$), which contradicts maximality of the chain $W_1\le W_2\le \cdots \le W_r$.

 It follows that} 
 there is an index $j<i$ {(which necessarily satisfies $e_j\in U_t$)} such that $a_{ij}\ne0$ {for some $A\in V$}.
  But then, using {the second part of} Lemma \ref{lm:Eij}(b), we conclude in particular that {the} diagonal matrix 
 {$E_{jj}-E_{ii}=[E_{ji},E_{ij}]$ is a} member of $V$. {Since $i$ was arbitrary such that $e_i\in U_t$ and $e_{i-1}\in U_t$, i}t follows that
{$V$ contains all those diagonal matrices with trace zero whose only possible nonzero entries correspond to the subspace $U_t$.}
{ General members of the corresponding diagonal block have maximal number of eigenvalues equal to $n_t=\dim U_t$.}
Let $A\in V$ be arbitrary and let $A_1, A_2, \ldots, A_r$ be the diagonal blocks of $A$ corresponding to the chain $W_1\le W_2\le \cdots \le W_r$. We will show that the matrix $A_1\oplus A_2\oplus \cdots \oplus A_{t-1} \oplus A_{t+1}\oplus \cdots\oplus A_r$ has at most $k-n_t$ distinct eigenvalues
. Assume the contrary. We have shown above that there exists a diagonal matrix $D\in V$ such that its $t$-th diagonal block has trace zero and $n_t$ nonzero distinct eigenvalues, while all the other blocks are zero. General linear combination of $A$ and $D$ then has at least $k+1$ distinct eigenvalues contradicting our assumption.
So, the matrix $A_1\oplus A_2\oplus \cdots \oplus A_{t-1} \oplus A_{t+1}\oplus \cdots\oplus A_r$ has at most $k-n_t$ distinct eigenvalues. Then it is clear that any linear combination of $A$ and $E_{ij}$, where $e_i,e_j\in U_t$, has at most $k$ distinct eigenvalues. By maximality of the dimension of $V$ it follows that $E_{ij}\in V$ concluding property \emph{(a)}.


 Property \emph{(b)} now follows easily by maximality of the dimension of $V$. 
\end{proof}

    In the situation of Lemma \ref{lm:triang} let $n_t=\dim U_t$ for $t$ such that $\dim W_{t-1}<\dim W_t-1$, and let $n_t=0$ otherwise. Furthermore, 
 let $l$ be the dimension of the space $V'$ of diagonal members of $V$ that correspond to the standard basis vectors which do not belong to any of the subspaces $U_i$} Note also that $V$ is invariant under projection on the diagonal by Lemma \ref{lm:triang}.

\begin{cor}\label{cor:k}
  Then $k\ge l+\displaystyle \sum_{t=1}^{r}n_t$.
\end{cor}

\begin{proof}

The maximal number of distinct eigenvalues of a member of $V$ is clearly a sum of all $n_t$ and the maximal number of distinct eigenvalues of a matrix from $V'$. The last number cannot be smaller than $l$, as $V'$ is $l$-dimensional.
\end{proof}

The proof of the main result of this section will be based on the above corollary and the following lemma.

\begin{lemma}\label{neenakost}
Let $k,l,r$ be positive integers and $n_1,\ldots ,n_r$ nonnegative integers satisfying $k\ge l+\sum_{t=1}^rn_t$. Then
$$l+\sum_{t=1}^r{n_t+1\choose 2} \le {k\choose 2}+1.$$
The equality holds if and only if $r=1$, $k=l+n_t$ for some $t$ and either $l=1$ or $k=l=2$.
\end{lemma}

\begin{proof}
The second entry on the left hand side, multiplied by 2, can be estimated
\[
    \sum_{t=1}^{r}n_t+\sum_{t=1}^{r}n_t^2\le \sum_{t=1}^{r}n_t \left(1+\sum_{t=1}^{r}n_t\right)\le (k-l)(k-l+1),
\]
{and} it is clear that the difference between the right hand side and the left hand side of the first inequality above is equal to
\[
    \sum_{\substack{s,t=1\\s\neq t}}^{r}n_sn_t
\]
which equals zero only if no more than one number $n_t$ is nonzero. Moreover, the equality in the second inequality holds only if $k=l+\sum_{t=1}^rn_t$.
It follows that
\begin{eqnarray*}
l+\sum_{t=1}^r{n_t+1\choose 2}&\le& \frac{1}{2}\left(k^2-2kl+l^2+k+l\right)={k\choose 2}-k(l-1)+\frac{l(l+1)}{2}\\
&\le &{k\choose 2}-\frac{1}{2}l^2+\frac{3}{2}l\le {k\choose 2}+1,
\end{eqnarray*}
where we used the fact that $1\le l\le k$. We have equality in the last two inequalities only if $l=1$ or $k=l=2$.
\end{proof}

{In the following theorem we prove 
{that $m=\displaystyle{n\choose 2}+{k\choose 2}+1$} which {proves the first part of Theorem \ref{thm:main} and }solves \cite[Conjecture~1.{2}]{LR}. In the theorem we also characterize subspaces {where equality holds in Theorem \ref{thm:main} and which additionally satisfy} conditions of Lemma \ref{lm:borel} when $k\ge 3$. We will use this characterization in Lemma \ref{lem:regular} which is one of the key steps in the characterization of 
\textit{all} subspaces {where the upper bound in Theorem \ref{thm:main} is achieved.}

\begin{thm}\label{prop:structure_borel}
{Let $1\le k\le n-1$ and let $m$ be the maximal possible dimension of a subspace of $M_n(\mathbb{C})$ whose all elements have at most $k$ distinct eigenvalues. Then
$$\displaystyle m={n\choose 2}+{k\choose 2}+1.$$
Moreover, if $k\ge 3$ and $V\in F_m(X)$ is a subspace satisfying conditions of Lemma \ref{lm:borel}, then {there exists $p\in \{0,1,\ldots ,n-k+1\}$ such that $V$ consists of all matrices
\begin{equation*}
\begin{pmatrix}
A & B & C\\0 & D & E\\0 & 0 & F
\end{pmatrix}
\end{equation*}
where $B\in M_{p\times (k-1)}(\mathbb{C})$, $D\in M_{k-1}(\mathbb{C})$ and $E\in M_{(k-1)\times (n-k-p+1)}(\mathbb{C})$ are arbitrary and
$\begin{pmatrix}
A & C\\0 & F
\end{pmatrix}$
is an arbitrary upper triangular matrix with equal diagonal entries.
}

}
\end{thm}

}

\begin{proof}

{The spaces of matrices described in the {theorem }
are clearly invariant under conjugation by invertible upper triangular matrices, they have dimension $\displaystyle{n\choose 2}+{k\choose 2}+1$ and {for each $k\ge 1$} they consist of matrices with at most $k$ distinct eigenvalues, so $\displaystyle m\ge {n\choose 2}+{k\choose 2}+1$. To prove the converse, {by the Borel fixed point theorem} it suffices to show that spaces satisfying conditions of Lemma \ref{lm:borel} have dimension at most $\displaystyle{n\choose 2}+{k\choose 2}+1$. Let $V$ be such a space} and let $l$ and $n_t$ be defined as before Corollary \ref{cor:k}. Note that $l$ is positive, as scalar matrices are in $V$ by the maximality of the dimension of $V$. 
 We compute 
\[
    \dim V = {l}+{n\choose 2}+\sum_{t=1}^{r}{n_t+1\choose2},
\]
where the first entry on the right hand side counts {the diagonal elements of $V$ corresponding to standard basis vectors that do not belong to any of the subspaces $U_t$,} 
 the second one counts the entries strictly above the diagonal, and the terms of the third one count the entries below the diagonal and on it corresponding to each of the subspaces $U_t$. 
The first part of the theorem now immeadiately follows from Lemma \ref{neenakost}. Moreover, the equality in the inequality of the lemma holds only if $r=1$, $k=l+n_t$ for some $t$ 
 and either $l=1$ or $k=l=2$. In particular, if $k\ge 3$, then $l=1$, $r=1$ and $n_t=k-1\ge 2$, which gives us the possibilities for $V$ described in the theorem.
\end{proof}

\section{{Preliminaries for the structure result }}
\label{sec3pre}

In our considerations we will often refer to properties that hold \textit{generically}. As usual in the algebraic geometry, this will mean that the property holds on an open dense subset (in the Zariski topology). Most often generic conditions will be considered on some line. In this case a property will hold generically if it will hold for all but finite number of points of the line.

In the proof of the second part of Theorem \ref{thm:main} we will need the following condition more than once.

\noindent \textbf{Two zeros condition:} Let $n$ be a positive integer, let $s$ be a polynomial of degree no more than $n$, and let $\lambda\in \mathbb{C}\setminus\{0\}$ be fixed. For 
generic $\mu\in\mathbb{C}$ polynomial
\[
    r_\mu(t)=t^{n+2}-\lambda t^{n+1}{-}\mu s(t) 
\]
has no more than two distinct zeros.

\begin{lemma}\label{le:twozeros}
  The two zeros condition implies that $s=0$. 
\end{lemma}

\begin{proof}
Let ${s}(t)= a_n t^n +a_{n-1}t^{n-1}+\cdots + a_0$.
 By the condition under consideration, ${r}_\mu$ has at most two distinct zeros for all but a finite number of scalars $\mu$. 
 \moblue{Besides}, this polynomial has a simple zero at $\mu=0$. According to Lemma \ref{open_conditions}(c) {(applied e.g.\ to the companion matrix of a polynomial)} the condition that \moblue{a monic polynomial of given degree}
 has a simple zero is open, so that ${r}_\mu$ has a simple zero 
generically, i.e., for all but a finite number of $\mu\in\mathbb{C}$. In the rest of the proof we consider such $\mu$ that $r_{\mu}$ has a simple zero and at most two distinct zeros. For all these values of $\mu$ we can write ${r}_\mu(t) =(t-\alpha)(t-\beta)^{n+1}$, where $\alpha$ and $\beta$ {may} depend on $\mu$. The first Vieta formula determines $\alpha$ as a linear function of $\beta$, i.e.,
  \[
    \alpha =  \lambda-(n+1)\beta.
  \]
  Insert this expression into the second and the third Vieta formula to get two polynomial conditions
\begin{equation}\label{Viete2}
{n+2\choose 2}\beta ^2-(n+1)\lambda\beta -\mu a_n
=0\quad \mathrm{and}
\end{equation}
\begin{equation}\label{Viete3}
2{n+2\choose 3}\beta ^3-{n+1\choose 2}\lambda\beta ^2+\mu a_{n-1} 
=0
\end{equation}
in $\beta$ and $\mu$. 
For a fixed $\mu\in \mathbb{C}$ \moblue{as chosen above} the polynomial equations \eqref{Viete2} and \eqref{Viete3} have a common solution $\beta$, so the resultant (i.e. the determinant of the Sylvester matrix)
$${\left|\begin{array}{ccccc}
{n+2\choose 2} & -(n+1)\lambda & -\mu a_n & 0 & 0\\
0 & {n+2\choose 2} & -(n+1)\lambda & -\mu a_n & 0\\
0 & 0 & {n+2\choose 2} & -(n+1)\lambda & -\mu a_n\\
2{n+2\choose 3} & -{n+1\choose 2}\lambda & 0 & \mu a_{n-1} & 0\\
0 & 2{n+2\choose 3} & -{n+1\choose 2}\lambda & 0 & \mu a_{n-1}
\end{array}
\right|}$$
of these two polynomials has to be zero. (Note that this is a special case of the result from \cite{KN} used in Lemma \ref{open_conditions}, or see some standard textbook on algebraic geometry such as \cite[Section 3.5]{CLOS}.)
Since this condition is satisfied for general $\mu$ \moblue{as considered above}, all the coefficients at powers of $\mu$ in the obtained {resultant} 
 must be zero. 
 {{Since $\mu$ appears only in the last three columns of the above determinant, the degree of the resultant is at most 3, and it has zero constant term, since all entries of the last column of the determinant are multiples of $\mu$.}}
 Now it is clear that the coefficient at $\mu$ equals
{\begin{align*}&\left|\begin{array}{ccccc}
{n+2\choose 2} & -(n+1)\lambda& 0 & 0 & 0\\
0 & {n+2\choose 2} & -(n+1)\lambda & 0 & 0\\
0 & 0 & {n+2\choose 2} & -(n+1)\lambda & -a_n\\
2{n+2\choose 3} & -{n+1\choose 2}\lambda & 0 & 0 & 0\\
0 & 2{n+2\choose 3} & -{n+1\choose 2}\lambda & 0 & a_{n-1}
\end{array}
\right|
=\frac{1}{2}{n+2\choose 3}(n+1)^3\lambda^3a_{n-1}
\end{align*}}
and the coefficient at $\mu^3$ equals
\begin{align*} &\left|\begin{array}{ccccc}
{n+2\choose 2} & -(n+1)\lambda & -a_n&0&0\\
0 & {n+2\choose 2} & 0 & -a_n&0\\
0 & 0 & 0 & 0 & -a_n\\
2{n+2\choose 3} & -{n+1\choose 2}\lambda & 0 & a_{n-1} & 0\\
0 & 2{n+2\choose 3} & 0 & 0 & a_{n-1}
\end{array}
\right|=-4{n+2\choose 3}^2a_n^3.
\end{align*}
It follows that $a_n=0$ and $a_{n-1}=0$.
  Using 
{\eqref{Viete2} and \eqref{Viete3}} one then concludes that $\beta=0$ and $\alpha= \lambda$ independently of the chosen $\mu$ and hence $s=0$.
\end{proof}

In the proof of Theorem \ref{thm:main} we will also need the fact that for $k\ge3$ a space $V$ of maximal dimension contains a \moblue{regular matrix that has exactly $k-1$ simple eigenvalues, i.e., it is} similar to
\begin{equation}\label{eq:regular_example}
\begin{pmatrix}
                                                 \lambda_1 &  &  &  &  &  &  \\
                                                 & \ddots &  &  &  &  &  \\
                                                 &  & \lambda_{k-1} &  &  &  & \\
                                                 &  &  & \lambda_k & 1 &  &  \\
                                                 &  &  &  & \ddots & \ddots & \\
                                                 &  &  &  &  & \ddots & 1 \\
                                                 &  &  &  &  &  & \lambda_k
                                               \end{pmatrix},
\end{equation}
where all the empty entries of this matrix are zeros and $\lambda_1,\lambda_2,\ldots ,\lambda_k$ are pairwise distinct.

\begin{lemma}\label{lem:regular}
 Let $3\le k< n$ and let $V$ be a subspace of $M_n(\mathbb{C})$ of dimension $m=\displaystyle {n\choose 2}+ {k\choose2}+1$ whose members have at most $k$ distinct eigenvalues. Then $V$ contains a regular element that has $k-1$ simple eigenvalues.
\end{lemma}

\begin{proof}
 Assume the contrary. Let $Y$ be the set of all matrices that are not regular, and let $Z$ be the set of all matrices with at most $k-2$ simple eigenvalues. By Lemma \ref{open_conditions} the sets $Y$ and $Z$ are both closed in the Zariski topology, so is their union. Moreover, the union $Y\cup Z$ is clearly invariant under the action of the group $\mathcal{T}_n$ of all invertible upper triangular matrices by conjugation. {Furthermore, the Zariski closed set $Y\cup Z$ is clearly homogeneous, so we may view it as a projective variety.} Hence, we can introduce the Fano variety
$$F_m(Y\cup Z)=\{W\in\mathrm{Gr}(m,n^2);W\subseteq Y\cup Z\}$$
of the union $Y\cup Z$. The assumption that $V$ does not contain a regular element that has $k-1$ simple eigenvalues implies that the intersection $F_m(X)\cap F_m(Y\cup Z)$ is not empty. This intersection is invariant under $\mathcal{T}_n$, so by the Borel fixed point theorem it has a fixed point $V'$. However, since $k\ge 3$, $V'$ is then one of the 
 spaces described in {Theorem} 
 \ref{prop:structure_borel} and it contains a regular element with $k-1$ simple eigenvalues, which is similar (with a similarity that swaps the first two block rows and columns of \eqref{eq:required_form}) to a matrix of the form \eqref{eq:regular_example} with $\lambda_i$ pairwise distinct, contradicting the starting assumption {on $V'$}.
\end{proof}

\section{{Structure of some special spaces}}
\label{sec3}

Throughout the rest of the paper let $k\ge 3$. 
We will prove the main result by induction on $k$. Let $V$ be a space of maximal dimension $m=\displaystyle {n\choose2}+{k\choose2}+1$ satisfying the conditions of Theorem \ref{thm:main}. Note that by maximality we may assume that $V$ contains all scalar matrices. First we make a reduction that is based on Lemma \ref{lem:regular}.
{Each regular $n\times n$ matrix with $k-1$ simple eigenvalues which has at most $k$ distinct eigenvalues is similar to a matrix of the form \eqref{eq:regular_example}, so we now conjugate the space $V$ by an appropriate invertible matrix to assume that $V$ contains a matrix of the form \eqref{eq:regular_example} for some distinct $\lambda _1,\ldots ,\lambda _k\in \mathbb{C}$.}

{
In this section we describe the structure of the space considered under the following additional assumption.

\textbf{AA:} \emph{If a matrix $\begin{pmatrix}
                           \alpha & b^T \\
                           c & D
                         \end{pmatrix}$ with blocks of respective sizes 1 and $n-1$ belongs to $V$, then \moblue{the} matrices $\begin{pmatrix}
                           0 & b^T \\
                           0 & 0
                         \end{pmatrix}$, $\begin{pmatrix}
                           0 & 0 \\
                           c & 0
                         \end{pmatrix}$, and $\begin{pmatrix}
                           \alpha & 0 \\
                           0 & D
                         \end{pmatrix}$ belong to $V$.
}

This additional assumption is motivated by Representation Theory. Denote by $\mathbb{C}^*$ the multiplicative group $(\mathbb{C}\setminus \{0\},\cdot)$.
Let ${\phi}:{\mathbb{C}^* \rightarrow\mathrm{GL}_n}$ be {a group homomorphism} defined by
$t\mapsto\begin{pmatrix}
           t &  0  \\
           0  &  I_{n-1}
         \end{pmatrix}$. {For each $t\in \mathbb{C}^*$ the space ${\phi} (t)V{\phi} (t)^{-1}$ is $m$-dimensional and with elements having at most $k$ distinct eigenvalues, so it belongs to $F_m(X)$.

         \begin{lemma}
           Condition \textbf{AA} is equivalent to ${\phi} (t)V{\phi} (t)^{-1}=V$ for all $t\in\mathbb{C}^*$.
         \end{lemma}

\begin{proof}
   Condition \textbf{AA} clearly implies  ${\phi} (t)V{\phi} (t)^{-1}=V$. Conversely, a space satisfying this equality is a ${\mathbb{C}^*}$-module for the action {$(t,A)\mapsto \phi(t)A\phi(t)^{-1}$.} {Now we use the fact that every $\mathbb{C}^*$-module is a direct sum of weight spaces (see e.g. \cite[Proposition 22.5.2(iii)]{TY}) to see that $V$} 
can be written as
\[
    V=\bigoplus_{j\in\mathbb{Z}}V(j),\quad\mbox{where}\quad V(j)=\{A\in V;{{}\phi}(t)A{{}\phi}(t)^{-1}=t^j A,\ \mbox{for all}\ t\ne0 \}.
\]
  Write a matrix $A\in V(j)$ as $A=\begin{pmatrix}
                                       \alpha & b^\mathrm{T} \\
                                       c & D
                                     \end{pmatrix}$ {where $\alpha \in \mathbb{C}$, $b,c\in \mathbb{C}^{n-1}$ and $D\in M_{n-1}(\mathbb{C})$}.
It follows easily that members of $V(j)$ are nonzero only in the case that $j=0$, $j=1$, or $j=-1$. Elements of $V(0)$ are of the form $A=\begin{pmatrix}
                                       \alpha & 0 \\
                                       0 & D
                                     \end{pmatrix}$, elements of $V(1)$ are of the form $A=\begin{pmatrix}
                                       0 & b^\mathrm{T} \\
                                       0 & 0
                                     \end{pmatrix}$, and elements of $V(-1)$ are of the form $A=\begin{pmatrix}
                                       0 & 0 \\
                                       c & 0
                                     \end{pmatrix}$, after a straightforward computation.
\end{proof}

{ Our next step will be to estimate the dimensions of $V(j)$ when $j=\pm 1$. Here is an additional notation we need to introduce.
Let $V'(1)$ be the set of all matrices of the form $\begin{pmatrix}
                                                        0 & b^T \\
                                                        0 & 0
                                                      \end{pmatrix}{\in V(1)}$ such that $b$ has first $k-2$ entries equal to zero.
We define similarly $V'(-1)$.


\begin{lemma}\label{le:dimen}
  Let $l\ge k$ be the smallest index with the property that some 
{{} row $b^T=\begin{pmatrix}
     0 & \cdots & 0 & b_k & \cdots &   b_n
   \end{pmatrix}$} {with $b_l\ne 0$} equals the upper-right corner of a member of {$V'(1)$ {(with the convention $l=n+1$ if $b$ is always zero)}.} 
 Furthermore, let $c=\begin{pmatrix} c_2\\\vdots\\c_n\end{pmatrix}$ be an arbitrary lower-left corner of a member of $V(-1)$. Then $c_q=0$ for all $q\ge l$.
\end{lemma}

\begin{proof}
Clearly we may assume $l\le n$.  Choose an arbitrary $\begin{pmatrix}
                                       0 & b^\mathrm{T} \\
                                       0 & 0
                                     \end{pmatrix}\in V'(1)$ with $b_{l}\ne0$  and $\begin{pmatrix}
                                       0 & 0 \\
                                       c & 0
                                     \end{pmatrix}\in V(-1)$.
As explained in the beginning of this section
$V$ contains a matrix of the form
{\eqref{eq:regular_example}} for some distinct $\lambda_1,\ldots ,\lambda_k\in \mathbb{C}$
\ so that
it contains

$$A({\mu})= \begin{pmatrix}
\lambda_1& &  &  & \mu b_k &\cdots  &\cdots  & \mu b_n \\
c_2&\ddots\\\vdots& &\ddots &  &  &  &  &  \\
c_{k-1}&&  & \lambda_{k-1} &  &  &  & \\
c_k&  &&  & \lambda_k & 1 &  &  \\
\vdots&  &  &&  & \ddots & \ddots & \\
\vdots&  &  &&  &  & \ddots & 1 \\
c_n&  &&  &  &  &  & \lambda_k
\end{pmatrix}$$
 for some distinct $\lambda_1,\ldots ,\lambda_k\in \mathbb{C}$ and arbitrary $\mu\in\mathbb{C}$. Our assumptions imply that this matrix has at most $k$ {distinct} eigenvalues so that {its characteristic} polynomial $\Delta(t)$, {which is computed in} Lemma \ref{le:char} {below,} has at most $k$ {distinct} zeros {for arbitrary $\mu\in \mathbb{C}$.
 As shown in Lemma  \ref{le:char} the polynomial $\Delta(t)$ is the product of $(\lambda_2-t)\cdots (\lambda_{k-1}-t)$ and a polynomial of the form ${r}_{\mu}(t)=(\lambda_1-t)(\lambda_k-t)^{n-k+1}+\mu {s}(t)$ for some polynomial ${s}$. By the assumption the numbers $\lambda_2,\ldots ,\lambda_{k-1}$ are not zeros of the polynomial ${r}_0$, hence they are not zeros of ${r}_{\mu}$ for generic 
 $\mu\in \mathbb{C}$.
 Consequently, the polynomial ${r}_{\mu}$ has at most two distinct zeros for generic 
 $\mu$. }
{Now we} replace $t$ by $t+\lambda_k$ 
{and divide by $(-1)^{n-k}$} to get the following. For generic 
 $\mu\in\mathbb{C}$ \moblue{the} polynomial
\[
    t^{n-k+2}-(\lambda_1-\lambda_k)t^{n-k+1}{-}\mu\sum_{i=0}^{n-k} \sum_{p={k}}^{{k+}i} b_pc_{n-k-i+p}\, t^i
\]
satisfies the Two zeros condition.
Using Lemma \ref{le:twozeros} we conclude that
\begin{equation}\label{eq:coefficients_two_zeros}
    \sum_{p={k}}^{{k+}i} b_pc_{n-k-i+p}=0
\end{equation}
for all $i={0,}1,\ldots,n-k$. {{}Recalling that $l$ is the smallest index with $b_l\ne 0$, t}he equalities {\eqref{eq:coefficients_two_zeros}} 
 imply that $c_q=0$ for all $q\ge l$ as desired.
\end{proof}

}

\begin{lemma}\label{le:char} $\Delta(t)=\det(A(\mu)-tI)=$
\[
  (\lambda_2-t)\cdots(\lambda_{k-1}-t)\left((\lambda_1-t)(\lambda_k-t)^{n-k+1}+ \mu\sum_{i=0}^{n-k}(-1)^{n-k+i+1}(\lambda_k-t)^i\sum_{p=k}^{k+i}b_pc_{n-k-i+p} \right).
\]
\end{lemma}

\begin{proof}
  First observe that {{the} columns} indexed by 
 $i=2,\ldots,k-1$ contain only one nonzero entry which equals $\lambda_i-t$. 
 Let us perform the usual {column} 
 expansions along all of these {{}columns} 
 consecutively to conclude 
 that
  \[
  \Delta(t)=(\lambda_2-t)\cdots(\lambda_{k-1}-t)\Delta_1(t),
  \]
  where
  \[
  \Delta_1(t)=\begin{vmatrix}
             \lambda_1-t & \mu b_k & \cdots & \cdots & \mu b_n \\
             c_k & \lambda_k-t & 1 &  & 0 \\
             \vdots & 0 & \lambda_k-t & \ddots &  \\
             \vdots & \vdots & \ddots & \ddots & 1 \\
             c_n & 0 & \cdots & 0 & \lambda_k-t
           \end{vmatrix}.
  \]
We compute $\Delta_1(t)$ by expanding it first along the first row and then along the first column. The final minor is possibly nonzero only in the case of the {($p-k+2$)}-th column and {($q-k+2$)}-th row, for
  $p,q=k,\ldots,n$, such that $p\le q$, in which case it equals $(\lambda_k-t)^{(n-k)-(q-p)}$. So,
$$\Delta_1(t)={(\lambda_1-t)(\lambda_k-t)^{n-k+1}+\mu}\sum_{k\le p\le q\le n}b_pc_q(-1)^{p+q+1} (\lambda_k-t)^{n-k+p-q}$$
 which gives the desired result after a small computation.
\end{proof}
}

\begin{cor}\label{cor:dimen}
$\dim V'(1)+\dim V(-1)\le n-1$.

\end{cor}

\begin{proof} 
The conclusion of Lemma \ref{le:dimen} implies easily the desired estimates.
\end{proof}

{
\begin{cor}\label{cor:dim}
  $\dim V(1)+\dim V(-1)\le n+k-3$
  .
\end{cor}

\begin{proof}
This follows immediately from Corollary \ref{cor:dimen}. Indeed, $\dim V(1)\le\dim V'(1) +k-2$ 
and the desired inequality follows. 
\end{proof}
}

{
We now recall that $\displaystyle\dim V= {n\choose2}+{k\choose2}+1$. It follows by Corollary \ref{cor:dim} that
$$\dim V(0)\ge \displaystyle {n\choose2}+{k\choose2}+1-n-k+3 ={n-1\choose2}+{k-1\choose2}+2.$$
 Recall that $V(0)$ is a linear space of matrices of the form $\begin{pmatrix}
   \alpha & 0 \\
   0 & D
 \end{pmatrix}$, whose lower-right corners form a space{, which we denote by $W$}, of dimension no smaller than $\displaystyle{n-1\choose2}+{k-1\choose2}+1$. On the other hand, members of \moblue{$W$
 } have no more than $k-1$ {distinct} eigenvalues. Indeed, if $D\in W$ had $k$ distinct eigenvalues, then some linear combination of the corresponding matrix $\begin{pmatrix} \alpha & 0 \\ 0 & D \end{pmatrix}\in V$ and some matrix of the form \eqref{eq:regular_example} would lie in $V$ and have at least $k+1$ distinct eigenvalues, a contradiction. {Using Theorem \ref{prop:structure_borel} we can therefore conclude that the dimension of $W$ is exactly $\displaystyle{n-1\choose 2}+{k-1\choose 2}+1$. 
 This fact implies that in all the above inequalities up to and including Corollary \ref{cor:dimen} we have equalities, more precisely:

 \begin{cor}\label{cor:dimension}
   \begin{enumerate}[(a)]
     \item $\dim\;V(0)=\displaystyle{n-1\choose2}+{k-1\choose2}+2$
     \item $\dim\;V(1)=n-l+k-1$
     \item $\dim\;V(-1)=l-2$
   \end{enumerate}
 \end{cor}

This proves that the space of lower-left corners of $V(-1)$ equals the span of $\{e_1,\ldots, e_{l-2}\}\moblue{\subseteq}\mathbb{C}^{n-1}$.
Using this result we now show a version of Lemma \ref{le:dimen} in which the roles of $b$ and $c$ are interchanged.

\begin{lemma}\label{lem:bji}
An arbitrary upper-right corner of a member of $V(1)$ is of the form ${b^T=\begin{pmatrix} b_2 & \cdots & b_{k-1} & 0 & \cdots & 0 & b_l & \cdots & b_n \end{pmatrix}}$.
\end{lemma}

 {\begin{proof}
    If $l>k$ and $\begin{pmatrix} 0 & b^T \\ 0 & 0 \end{pmatrix} \in V(1)$ is arbitrary, then a matrix
$$\begin{pmatrix}
\lambda_1 & b_2 & \cdots & b_{k-1} & b_k & \cdots & b_{l-1} & b_l & \cdots & b_n \\
0 & \lambda _2\\
\vdots & & \ddots\\
0 & & & \lambda_{k-1}\\
0 & & & & 0 & 1\\
\vdots & & & & & \ddots & \ddots \\
\mu & & & & & & \ddots & \ddots \\
0 & & & & & & & \ddots & \ddots \\
\vdots & & & & & & & & \ddots & 1\\
0 & & & & & & & & & 0
\end{pmatrix},$$
 where $\lambda _1,\ldots ,\lambda _{k-1}$ are nonzero and pairwise distinct and $\mu$ appears in the $(l-1)$-st row, belongs to $V$ for all $\mu \in \mathbb{C}$. The characteristic polynomial $\Delta(t)$ of this matrix equals $\Delta (t)=(\lambda_2-t)\cdots (\lambda_{k-1}-t)(-t)^{n-l+1}\cdot \Delta_1(t)$ where
$$\Delta_1(t)=\begin{vmatrix}
\lambda _1-t & b_k & \cdots & \cdots & b_{l-1}\\
0 & -t & 1\\
\vdots & & \ddots & \ddots \\
0 & & & \ddots & 1 \\
\mu & & & & -t
\end{vmatrix}=(\lambda_1-t)(-t)^{l-k}+(-1)^{l-k}\mu \sum_{i=0}^{l-k-1}b_{i+k}t^i.$$
Since the matrix defined above belongs to $V$, it has at most $k$ distinct eigenvalues, and as in Lemma \ref{le:dimen} we conclude that \moblue{the} polynomial $(-t)^{n-l+1}\Delta_1(t)$ has at most two distinct zeros for generic 
 $\mu \in \mathbb{C}$. Lemma \ref{le:twozeros} (applied to $(-t)^{n-l+1}\Delta_1(t)$) then implies that $\Delta_1(t)=(\lambda_1-t)(-t)^{l-k}$, so $b_i=0$ for $i=k,\ldots ,l-1$, as desired.
  \end{proof}

Recall that $W\subseteq M_{n-1}(\mathbb{C})$ is a subspace of dimension $\displaystyle {n-1\choose2}+{k-1\choose2}+1$ whose members have at most $k-1$ distinct eigenvalues.
 If $k\ge 4$}, it now follows from the inductive hypothesis that {there exists $p\in \{0,1,\ldots ,n-k+1\}$ such that} {the} members {of $W$} are simultaneously similar to matrices of the form \eqref{eq:required_form}, i.e. $\begin{pmatrix} A & B & C\\0 & D & E\\0 & 0 & F \end{pmatrix}$, with 
blocks of respective sizes $p,k-2,n-p-k+1$, where $\begin{pmatrix}
                        A & C \\
                        0 & F
                      \end{pmatrix}$ is upper triangular with constant diagonal, but other than that the {nonzero} blocks are arbitrary. {If $k=3$, then there are more similarity classes of (${n-1\choose 2}+2$)-dimensional spaces of $(n-1)\times (n-1)$ matrices having at most 2 eigenvalues, see \cite[Theorem 1.8]{dSP}.
                      However, the space $W$} contains the {$(n-1)\times (n-1)$} lower-right corner of the matrix given {by \eqref{eq:regular_example}, which has a simple eigenvalue. This additional information together with \cite[Theorem 1.8]{dSP} implies that the members of $W$ are simultaneously similar to matrices of the form \eqref{eq:required_form} even if $k=3$.

The next step is to prove that $W$ is actually {\em{equal}} to the space of matrices obtained from \eqref{eq:required_form} by interchanging the first two block rows and columns.
{

\begin{lemma}\label{lm:SE_corners_ofV_0}
The space $W$ of all $(n-1)\times (n-1)$ lower-right corners of $V(0)$ is equal to the space of all matrices of the form
$\begin{pmatrix} D & 0 & E\\B & A & C\\0 & 0 & F \end{pmatrix}$ with blocks of respective sizes $k-2$, $p$ and $n-p-k+1$ for some $p\in \{0,1,\ldots ,n-k+1\}$
where $\begin{pmatrix}A & C \\ 0 & F\end{pmatrix}$ is the sum of a scalar matrix and a strictly upper triangular matrix, and all the other nonzero blocks are arbitrary. 
\end{lemma}}

\begin{proof}
 Recall that the space $W$ is simultaneously similar to the space of matrices described in the lemma. Let $P$ be an invertible matrix that provides this similarity.
 The space $W$ contains the lower-right $(n-1)\times(n-1)$ corner of a {{}regular} matrix of the form \eqref{eq:regular_example} 
 such that
 \[ P\ \begin{pmatrix}
                                                 \lambda_2 &  &  &  &  &  &  \\
                                                 & \ddots &  &  &  &  &  \\
                                                 &  & \lambda_{k-1} &  &  &  & \\
                                                 &  &  & \lambda_k & 1 &  &  \\
                                                 &  &  &  & \ddots & \ddots & \\
                                                 &  &  &  &  & \ddots & 1 \\
                                                 &  &  &  &  &  & \lambda_k
                                               \end{pmatrix}=\begin{pmatrix} D & 0 & E\\B & A & C\\0 & 0 & F \end{pmatrix}P.
 \]
 Here, $A$ and $F$ are upper triangular with the same constant, say $\lambda$, on the diagonal. Denote this matrix of the form \eqref{eq:regular_example} by $L$ and the $3\times3$ block matrix on the right by $M$. Since the two matrices are similar and $\lambda_k$ is the only multiple eigenvalue of $L$ and $\lambda$ is a multiple eigenvalue of $M$, we have that $\lambda_k=\lambda$. Consequently, the eigenvalues of $D$ are $\lambda_2,\ldots,\lambda_{k-1}$. Write $P$ with blocks of respective sizes $k-2,p,n-p-k+1$ as $P=\begin{pmatrix}
                                                                 Q & R & S \\
                                                                 {N} & U & T \\
                                                                 X & Y & Z
                                                               \end{pmatrix}$ to get
 \begin{equation}\label{eq:similarity}
    \begin{pmatrix}
    Q & R & S \\
    {N} & U & T \\
    X & Y & Z
    \end{pmatrix}\begin{pmatrix}
                   D' & 0 & 0 \\
                   0 & \lambda I+J_1 & E_{p1} \\
                   0 & 0 & \lambda I+J_2
                 \end{pmatrix}=\begin{pmatrix} D & 0 & E\\B & A & C\\0 & 0 & F
                 \end{pmatrix} \begin{pmatrix}
    Q & R & S \\
    {N} & U & T \\
    X & Y & Z
    \end{pmatrix},
 \end{equation}
 where $D'=\mathrm{Diag}(\lambda_2,\ldots,\lambda_{k-1})$ { and $J_1,J_2$ are nilpotent Jordan blocks of appropriate sizes}.
 The $(3,1)$-block of equation \eqref{eq:similarity} equals $XD'=FX$. Since the intersection of the spectra of $D'$ and $F$ is empty, we conclude that $X=0$ \cite[Section VIII.1]{Gant}. We rewrite the $(3,2)$\moblue{-} and $(3,3)$\moblue{-}block of equation \eqref{eq:similarity} into
 \[
   \begin{pmatrix}
     Y & Z
   \end{pmatrix} J=(F-\lambda I) \begin{pmatrix}
     Y & Z
   \end{pmatrix},
 \]
 where $J$ is a {nilpotent} Jordan block of appropriate size. It follows inductively on $l$ that
 \[
   \begin{pmatrix}
     Y & Z
   \end{pmatrix} J^l=(F-\lambda I)^l \begin{pmatrix}
     Y & Z
   \end{pmatrix}.
 \]
 So, the matrix $\begin{pmatrix}
     Y & Z
   \end{pmatrix}$ maps $\mathrm{Im}\;J^l$ into $\mathrm{Im}\;(F-\lambda I)^l$ which is included in $\mathrm{Span}\{e_1,\ldots, e_{n-p-k+1-l}\}$ for all positive integers $l$.
   This readily yields that $Y=0$ and that $Z$ is upper triangular.

 Next, we consider the $(1,2)$\moblue{-}block of equation \eqref{eq:similarity} to get $R(\lambda I+ J_1)=DR$. Since the spectrum of $D$ does not contain $\lambda$, we determine that $R=0$. Block equation $(2,2)$ now implies that $UJ_1=(A-\lambda I)U$. 
 As above we deduce that $U$ is upper triangular. Finally, we conclude that $W$ is the space of all matrices of the form
 \[
   \begin{pmatrix} D & 0 & E\\B & \lambda I+ A' & C\\0 & 0 & \lambda I+ F'
                 \end{pmatrix},
 \]
 where $A'$ and $F'$ are strictly upper triangular. Indeed, we know that space $W$ \moblue{is} simultaneously similar to \moblue{-the space of} matrices of this form, while we proved here that \moblue{a} similarity matrix is also of this block form with $(2,2)$\moblue{-} and $(3,3)$\moblue{-}blocks upper triangular.
\end{proof}
{
{
Let us now write matrices with respect to the block partition of respective sizes $1$, $k-2$, $p$ and $n-p-k+1$. Then { it follows from the above lemma and the equality $\dim V(0)=\displaystyle{n{-1}\choose 2}+{k{-1}\choose 2}+2=\dim W+1$ that} $V(0)$ consists of all matrices of the form
\begin{equation}\label{eq:form}
  \begin{pmatrix}
             \alpha & 0  & 0  &  0 \\
             0  & D & 0 & E \\
             0  & B & \lambda I+A' & C \\
             0  & 0 & 0 & \lambda I+F'
           \end{pmatrix},
\end{equation}
where $\alpha,\lambda, D, E, B, C$ are arbitrary and $A', F'$  are strictly upper triangular. Also, $V(1)$ respectively $V(-1)$ consists of some matrices of the form
\[
    \begin{pmatrix}
      0 & b'^{\mathrm{T}} & b''^{\mathrm{T}}  & b'''^{\mathrm{T}}  \\
      0 & 0 & 0 & 0 \\
      0 & 0 & 0 & 0 \\
      0 & 0 & 0 & 0
    \end{pmatrix}\quad\mbox{respectively}\quad \begin{pmatrix}
      0 & 0 & 0 & 0 \\
      c' & 0 & 0 & 0 \\
      c'' & 0 & 0 & 0 \\
      c''' & 0 & 0 & 0
    \end{pmatrix}.
\]
{ We now determine the structure of the spaces $V(1)$ and $V(-1)$.}
{
\begin{lemma}\label{le:22}
  For any matrix in $V$ the blocks $b''$ and $c'''$ are zero.
\end{lemma}

\begin{proof}
  Here is a simplified notation for the first row and column that will be useful
\[
    b=\begin{pmatrix}
        b' \\
        b'' \\
        b'''
      \end{pmatrix}\quad\mbox{and}\quad c=\begin{pmatrix}
        c' \\
        c'' \\
        c'''
      \end{pmatrix}.
\]
{As in Lemma \ref{le:dimen} let $l\ge k$ be the smallest index such that some row $b^T=\begin{pmatrix}
     0 & \cdots & 0 & b_k & \cdots &   b_n
   \end{pmatrix}$
with $b_l\ne 0$ equals the upper-right corner of a member of $V'(1)$, with the convention $l=n+1$ if $V'(1)$ is trivial.}
By that lemma {the entries of any lower-left corner $c=\begin{pmatrix} c_2 \\ \vdots \\ c_n \end{pmatrix}$ of a member of $V(-1)$ satisfy} 
 $c_q=0$ for all 
 $q\ge l$. 
First, we want to show that
\[
    l\ge k+p. \quad\quad\quad\quad (*)
\]
Towards a contradiction we assume that $l<k+p$. {In particular, we have $p>0$ and $l\le n$.} Choose a matrix of the form \eqref{eq:form} with $\alpha=\lambda_1$, $D=\mathrm{Diag}(\lambda_2,\ldots,\lambda_{k-1})$, where $\lambda_i$'s are {nonzero and} distinct, $B=E_{l-k+1,s{-1}}$ { for some $s$, $2\le s\le k-1$}, and all the other blocks are zero. We add $c$ and $\mu b$ {described above} to this matrix and compute the resulting characteristic polynomial
\[
    \Delta(t)=\begin{vmatrix}
\lambda_1-t & 0 & \cdots &  0 &0&\cdots&0 & \mu b_l & \cdots &  \mu b_n \\
c_2 & \lambda_2-t &&&&&&&& \\
\vdots&& \ddots &  & &  &  & &&  \\
c_{k-1}&&  & \lambda_{k-1}-t & &&& &  &   \\
c_k&&  &  & -t & &&& &    \\
\vdots&&&&& \ddots&&&& \\
c_{l-1}&&&&&& -t&&&\\
0&& 1 &  &  &&& \ddots & & \\
\vdots&&&&  &  &  & &  \ddots  &\\
0&&&&  &  &  &  &  & -t
\end{vmatrix};
\]
observe that the matrix under this determinant belongs to $V$.
The isolated entry $1$ was set in the $l$-th row and the $s$-th column. 
 We expand the determinant at all rows and columns that contain only one nonzero entry:
\[
    \Delta(t)=(-t)^{n-k} \prod_{\substack{i=2\\i\neq s}}^{k-1} (\lambda_i-t)\cdot \begin{vmatrix}
                \lambda_1-t & 0 &  \mu b_l \\
                c_s & \lambda_s-t & 0 \\
                0 &   1   & -t
              \end{vmatrix}=(-t)^{n-k}q_\mu(t),
\]
where we introduce
\[
    q_\mu(t)=\prod_{\substack{i=2\\i\neq s}}^{k-1}(\lambda_i-t)\ \ ((\lambda_1-t)(\lambda_s-t)(-t)+\mu b_l c_s).
\]
So, if $\mu=0$, then $q_\mu$ has $k$ distinct zeros. Consequently, by Lemma \ref{open_conditions}(b) the polynomial $q_\mu$ has $k$ distinct zeros for generic $\mu$. If $q_\mu(0)\ne0$ for such $\mu$, then the polynomial $\Delta$ has $k+1$ zeros, a contradiction with the standing assumption on $V$. Therefore, generically the polynomial $q_\mu$ has $k$ distinct roots, one of which is zero. 
This implies that $b_lc_s=0$, and since $b_l\neq0$ we have $c_s=0$. Now, in this consideration $s$ is chosen arbitrary from the set $\{2,\ldots,k-1\}$, so $c_2,\ldots,c_{k-1}$ are all equal to zero.

We have shown that any first column of a member of $V(-1)$ is of the form $c=\begin{pmatrix}
                                                                                 0 & \cdots & 0 & c_k & \cdots & c_{l-1} & 0 & \cdots & 0
                                                                               \end{pmatrix}^T$\moblue{,}
 so that $\dim V(-1)\le l-k$\moblue{,} 
 contradicting Corollary \ref{cor:dimension}. 
 This shows that 
 Condition $(*)$  $l\ge k+p$ holds.

Next, we want to show that $l=k+p$.
Choose $c$ with $1$ in the $(l-2)$-th position (i.e. $c_{l-1}=1$) and zeros elsewhere. Assume towards a contradiction that $l>k+p$ {(and hence $p\ne n-k+1$ and $l>k$)} and repeat the above arguments with the roles of ${b}$ and ${c}$ interchanged. Consider a member of $V$ { of the form \eqref{eq:form} with $\alpha =\lambda_1$ and $D=\mathrm{Diag}(\lambda_2,\ldots ,\lambda_{k-1})$ where} $\lambda_1, \ldots, \lambda_{k-1}$ { are nonzero and pairwise distinct, and } with a $1$ in the $(l-1)$-st column and $s$-th row { for $s\in \{2,\ldots ,k-1\}$}, and with zeros everywhere else. 
Add to this matrix $c$ and $\mu b^T$ where $b^T$ is an arbitrary upper-right corner of $V(1)$ and $c$ is as above.
Recall that $b_k=\cdots=b_{l-1}=0$ by Lemma \ref{lem:bji}.
Computations as above reveal that the characteristic polynomial of this matrix is equal to
\[
	\Delta(t)=(-t)^{n-k} \prod_{\substack{i=2,\\i\neq s}}^{k-1}(\lambda_i-t)\cdot \begin{vmatrix}
                \lambda_1-t &  \mu b_s & 0  \\
                0 & \lambda_s-t & 1 \\
                1 &   0   & -t
              \end{vmatrix},
\]
where the last determinant on the right hand side equals $(\lambda_1-t)(\lambda_s-t)(-t)+\mu b_s$. As before we conclude that $b_s=0$ for all possible $s\in\{2,\ldots,k-1\}${, which implies $\dim V\le \displaystyle{n\choose 2}+{k\choose 2}-k+3$}. The contradiction so obtained brings us to the fact that $l=k+p$, which proves the lemma.
\end{proof}
}

The above lemma concludes the proof that a space $V\in F_m(X)$ containing a matrix of type \eqref{eq:regular_example} and satisfying condition {\bf AA} consists of all matrices of the form
$\begin{pmatrix}
D & 0 & E\\
B & A & C\\
0 & 0 & F
\end{pmatrix}$ with blocks of respective sizes $k-1$, $p$ and $n-p-k+1$, where
$\begin{pmatrix}
A & C\\
0 & F
\end{pmatrix}$ is upper triangular matrix with equal diagonal entries and all the other nonzero blocks are arbitrary.

{
\section{Structure of the spaces of maximal dimension}\label{sec:max}

In this section
we will prove the second part of Theorem \ref{thm:main}, i.e., the following theorem.

\begin{thm}
  Let {{}$3\le k<n$ and let }$V$ be a {subspace of $M_n(\mathbb{C})$ of dimension $m=\displaystyle {n\choose 2}+{k\choose 2}+1$ such that each member of $V$ has at most $k$ distinct eigenvalues.} 
 Then 
 there exists $p\in \{0,1, \ldots ,n-k+1\}$ such that $V$ is {{}simultaneously} similar to the space of all matrices of the form \eqref{eq:required_form} where $B\in M_{p\times (k-1)}(\mathbb{C})$, $D\in M_{k-1}(\mathbb{C})$ and $E\in M_{(k-1)\times (n-k-p+1)}(\mathbb{C})$ are arbitrary and
  $\begin{pmatrix} A & C\\0 & F \end{pmatrix}$ is an arbitrary upper triangular matrix with equal diagonal entries.
\end{thm}

Let $V$ be a space satisfying the conditions of the theorem.
Recall from the beginning of the previous section that we may assume that $V$ contains a matrix of the form \eqref{eq:regular_example} for some distinct $\lambda _1,\ldots ,\lambda_k\in \mathbb{C}$. 
Consider a block partition of $V$ with respect to dimensions $1$ and $n-1$. We define the projection $\pi_0$ from $V$ to the lower-left corner as $\pi_0:\begin{pmatrix}
                                                                                      \alpha & b^T\\
                                                                                      c & D
                                                                                    \end{pmatrix}
                                                                                    \mapsto \begin{pmatrix}
                                                                                     0  & 0\\
                                                                                     c & 0
                                                                                    \end{pmatrix}$. Clearly, $\ker\pi_0$ consists of all members of $V$ of the form $\begin{pmatrix}
                                                                                      \alpha & b^T\\
                                                                                      0 & D
                                                                                    \end{pmatrix}$.
Next we define the projection $\pi_0'$ from $\ker\;\pi_0$ to the diagonal blocks  as $\pi_0':\begin{pmatrix}
                                                                                      \alpha & b^T\\
                                                                                      0 & D
                                                                                    \end{pmatrix}
                                                                                    \mapsto \begin{pmatrix}
                                                                                     \alpha  & 0\\
                                                                                     0 & D
                                                                                    \end{pmatrix}$. Now, $\ker\pi_0'$ consists of all members of $V$ of the form $\begin{pmatrix}
                                                                                      0& b^T\\
                                                                                      0 & 0
                                                                                    \end{pmatrix}$.
Let
\[
  V_0=\mathrm{im}\, \pi_0\oplus\mathrm{im}\, \pi'_0\oplus\ker\pi_0'.
\]
It is clear that $\dim V_0=\dim V$.

\begin{lemma}\label{le:v0}
  Elements of $V_0$ have at most $k$ distinct eigenvalues.
\end{lemma}

\begin{proof}
  Let $\begin{pmatrix}
                                                                                      \alpha & b^T\\
                                                                                      c & D
                                                                                    \end{pmatrix}$
  be an arbitrary matrix in $V_0$. Then $\begin{pmatrix}
                                                                                      0& b^T\\
                                                                                      0 & 0
                                                                                    \end{pmatrix}\in \ker \pi_0'\subseteq V$ and
  $\begin{pmatrix}
                                                                                      \alpha & 0\\
                                                                                      0 & D
                                                                                    \end{pmatrix}\in \mathrm{im}\, \pi_0'$.
  So there exists $b'\in\mathbb{C}^{n-1}$ such that $\begin{pmatrix}
                                                                                      \alpha & b'^T\\
                                                                                      0 & D
                                                                                    \end{pmatrix}\in \ker \pi_0\subseteq V$.
Finally, $\begin{pmatrix}
                                                                                      0 & 0\\
                                                                                      c & 0
                                                                                    \end{pmatrix}\in \mathrm{im}\,\pi_0$,
so there exist $\alpha''\in\mathbb{C}, b''\in \mathbb{C}^{n-1}$, and $D''\in M_{n-1}(\mathbb{C})$ such that $\begin{pmatrix}
                                                                                      \alpha'' & b''^T\\
                                                                                      c & D''
                                                                                    \end{pmatrix}\in V$.
Members of $V$ have at most $k$ distinct eigenvalues, therefore the matrix
\[
\begin{split}
     & \phantom{=} \begin{pmatrix}
                                                                                      t & 0\\
                                                                                      0 & I
                                                                                    \end{pmatrix} \left(t \begin{pmatrix}
                                                                                      \alpha'' & b''^T\\
                                                                                      c & D''
                                                                                    \end{pmatrix}+\begin{pmatrix}
                                                                                      \alpha & b'^T\\
                                                                                      0 & D
                                                                                    \end{pmatrix}+t^{-1} \begin{pmatrix}
                                                                                      0 & b^T\\
                                                                                      0 & 0
                                                                                    \end{pmatrix}\right)\begin{pmatrix}
                                                                                      t^{-1} & 0\\
                                                                                      0 & I
                                                                                    \end{pmatrix}\\
     & = \begin{pmatrix}
                                                                                      \alpha+t\alpha'' & b^T+tb'^T+t^2 b''^T \\
                                                                                      c & D+tD''
                                                                                    \end{pmatrix}
\end{split}
\]
has at most $k$ distinct eigenvalues for each $t\neq0$. Consequently, the starting matrix has at most $k$ distinct eigenvalues by Lemma \ref{open_conditions}(b).
\end{proof}

 Next, we define the projection $\pi_\infty$ from $V$ to the upper-right corner as $\pi_\infty:\begin{pmatrix}
                                                                                      \alpha & b^T\\
                                                                                      c & D
                                                                                    \end{pmatrix}
                                                                                    \mapsto \begin{pmatrix}
                                                                                     0  & b^T\\
                                                                                     0 & 0
                                                                                    \end{pmatrix}$. Clearly, $\ker\pi_\infty$ consists of all members of $V$ of the form $\begin{pmatrix}
                                                                                      \alpha & 0\\
                                                                                      c & D
                                                                                    \end{pmatrix}$.
We define the projection $\pi_\infty'$ from $\ker\;\pi_\infty$ to the diagonal blocks  as $\pi_\infty':\begin{pmatrix}
                                                                                      \alpha & 0\\
                                                                                      c & D
                                                                                    \end{pmatrix}
                                                                                    \mapsto \begin{pmatrix}
                                                                                     \alpha  & 0\\
                                                                                     0 & D
                                                                                    \end{pmatrix}$. Now, $\ker\pi_\infty'$ consists of all members of $V$ of the form $\begin{pmatrix}
                                                                                      0& 0\\
                                                                                      c & 0
                                                                                    \end{pmatrix}$.
Let
\[
  V_\infty=\mathrm{im}\, \pi_\infty\oplus\mathrm{im}\, \pi'_\infty\oplus\ker\pi_\infty',
\]
so that again $\dim V_\infty=\dim V$. Similar arguments as in the proof of Lemma \ref{le:v0} show that all members of $V_\infty$ have at most $k$ distinct eigenvalues. Let us point out that the so defined spaces $V_0$ and $V_\infty$ satisfy condition \textbf{AA} from the beginning of Section \ref{sec3}.

\textbf{Remark.} Motivation for the definition of spaces $V_0$ and $V_\infty$ comes from Algebraic Geometry and Representation Theory. Recall the group homomorphism 
 ${\phi}:{\mathbb{C}^* \rightarrow\mathrm{GL}_n}$ 
 defined by
$t\mapsto\begin{pmatrix}
           t &  0  \\
           0  & I
         \end{pmatrix}$. {For each $t\in \mathbb{C}^*$ the space ${\phi} (t)V{\phi} (t)^{-1}$ is $m$-dimensional and with elements having at most $k$ distinct eigenvalues, so it belongs to the Fano variety $F_m(X)$. However, $F_m(X)$ is a projective variety, hence there exist limits} 
$$\displaystyle V_0=\lim_{t\rightarrow 0} {{}\phi}(t)V{{}\phi}(t)^{-1}\quad \mathrm{and}\quad \displaystyle V_{\infty}=\lim_{t\rightarrow 0} {{}\phi}(t)^{-1}V {{}\phi}(t)$$
{in $F_m(X)$}.  
A short computation reveals that $V_0$ is ${{}\phi}$-stable; indeed,
$$\displaystyle 
{{}\phi}(s)V_{0}{{}\phi}(s)^{-1}={{}\phi}(s)\lim_{t\rightarrow 0} {{}\phi}(t)V{{}\phi}(t)^{-1} {{}\phi}(s)^{-1}=\lim_{t\rightarrow 0} {{}\phi}(st)V{{}\phi}(st)^{-1}=V_0.$$ The same considerations apply to $V_{\infty}$.
Therefore, both spaces are ${\mathbb{C}^*}$-modules for the action 
{$(t,A)\mapsto \phi(t)A\phi(t)^{-1}$.} Observe that this is equivalent to Condition \textbf{AA}.

\textbf{Remark.} To show that the two definitions of $V_0$ are equivalent choose the following basis of the space $V$
 : $\begin{pmatrix}    \alpha_1 & b_1^{\mathrm{T}}  \\    c_1 & D_1  \end{pmatrix},\cdots,$ $\begin{pmatrix}    \alpha_r & b_r^{\mathrm{T}}  \\    c_r & D_r  \end{pmatrix},$ $\begin{pmatrix}\alpha_{r+1} & b_{r+1}^{\mathrm{T}} \\ 0 & D_{r+1} \end{pmatrix},\cdots,$ $\begin{pmatrix}    \alpha_{r+s} & b_{r+s}^{\mathrm{T}}  \\ 0 & D_{r+s} \end{pmatrix},$ $\begin{pmatrix}0 & b_{r+s+1}^{\mathrm{T}} \\ 0 & 0 \end{pmatrix},\cdots,$ $\begin{pmatrix} 0 & b_{m}^{\mathrm{T}}  \\ 0 & 0 \end{pmatrix}.$ Here, $r$ and $s$ are chosen consecutively the maximal possible so that $c_1,\ldots,c_r$ are linearly independent and that $\begin{pmatrix}\alpha_{r+1} & 0 \\ 0 & D_{r+1} \end{pmatrix}, \cdots,\begin{pmatrix} \alpha_{r+s} & 0  \\ 0 & D_{r+s} \end{pmatrix}$ are linearly independent.
  We denote the basis elements 
   by $B_1,\ldots,B_m$
   . Recall that $V$ is represented in $\mathrm{Gr}(m,n^2)\subseteq \mathbb{P}\left(\wedge^m(\mathbb{C}^{n^2})\right)$ by the class $\left[\bigwedge_{i=1}^mB_i\right]${, which is independent of the choice of the basis} (cf.\ \cite[Chapter 6]{Har}).  In order to get the basis of space $V_0$, we compute the limits of classes within the Grassmanian determined by exterior {products} 
 of basis elements:
  \[
    \begin{split}
       \left[ V_0 \right] & = \lim_{t\rightarrow0} \left[ \bigwedge_{i=1}^m {\phi}(t)B_i{\phi}(t)^{-1} \right] = \lim_{t\rightarrow0} \left[ \bigwedge_{i=1}^m  \begin{pmatrix}
                                           \alpha_i & tb_i^T \\
                                           t^{-1}c_i & D_i
                                         \end{pmatrix}   \right] \\
         & =\lim_{t\rightarrow0} \left[ \bigwedge_{i=1}^r  \begin{pmatrix}
                                           t\alpha_i & t^2b_i^T \\
                                           c_i & tD_i
                                         \end{pmatrix}  \wedge  \bigwedge_{i=r+1}^{r+s} \begin{pmatrix}
                                           \alpha_i & tb_i^T \\
                                           0 & D_i
                                         \end{pmatrix}  \wedge  \bigwedge_{i=r+s+1}^m  \begin{pmatrix}
                                           0 & b_i^T \\
                                           0 & 0
                                         \end{pmatrix}
                                          \right]\\
         & = \left[ \bigwedge_{i=1}^r  \begin{pmatrix}
                                           0 & 0\\
                                           c_i & 0
                                         \end{pmatrix}  \wedge  \bigwedge_{i=r+1}^{r+s} \begin{pmatrix}
                                           \alpha_i & 0 \\
                                           0 & D_i
                                         \end{pmatrix}  \wedge  \bigwedge_{i=r+s+1}^m  \begin{pmatrix}
                                           0 & b_i^T \\
                                           0 & 0
                                         \end{pmatrix}
                                          \right].
    \end{split}
  \]
Note that the elements of the above exterior product are indeed linearly independent, so a basis of $V_0$ 
 is given by
 \begin{equation*}\label{eq:basisV_0}
\begin{pmatrix}           0 & 0 \\      c_1 & 0    \end{pmatrix}, \cdots, \begin{pmatrix}      0 & 0 \\         c_r & 0   \end{pmatrix}, \begin{pmatrix}    \alpha_{r+1} & 0 \\     0 & D_{r+1}    \end{pmatrix}, \cdots, \begin{pmatrix}   \alpha_{r+s} & 0 \\    0 & D_{r+s}     \end{pmatrix}, \begin{pmatrix}   0 & b_{r+s+1}^T \\   0 & 0  \end{pmatrix},\cdots,\begin{pmatrix}  0 & b_m^T \\ 0 & 0  \end{pmatrix}.
\end{equation*}
Note that the first $r$ elements form a basis of $\mathrm{im}\,\pi_0$, the next $s$ elements form a basis of $\mathrm{im}\,\pi_0'$, and the rest of the elements form a basis of $\ker\pi_0'$. So, the two definitions of $V_0$ are equivalent. The same considerations apply to $V_\infty$.

Recall that \moblue{the} spaces $V_0$ and $V_\infty$ satisfy Condition \textbf{AA}, so we may apply the results of Section \ref{sec3}. 
In the block partition with respect to blocks of sizes {1,} $k-2$, $p$ and $n-p-k+1$ the upper-right corner of linear space $V_0(1)$ respectively $V_0'(1)$ is made of vectors of the form {$\begin{pmatrix}  {b'}^T & 0 & {b'''}^T  \end{pmatrix}$} respectively {$\begin{pmatrix}  0 &  0 &  {b'''}^T  \end{pmatrix}$}. Also, \moblue{the} {lower-left corner of} \moblue{the} linear space $V_0(-1)$ respectively $V_0'(-1)$ is made of all vectors of the form $\begin{pmatrix} c' \\ c'' \\ 0 \end{pmatrix}$ respectively $\begin{pmatrix}  0 \\  c'' \\ 0 \end{pmatrix}$. So, \moblue{the} linear space $V_0$ consists of all matrices of the form $\begin{pmatrix}
   \alpha & b'^{\mathrm{T}} & 0 & b'''^{\mathrm{T}} \\
   c' & D & 0 & E \\
   c'' & B & \lambda I+A' & C \\
   0 & 0 & 0 & \lambda I+F'
 \end{pmatrix}$, where $A'$ and $F'$ are strictly upper triangular. The case of the space $V_\infty$ goes in the same way. However, the block division there may be based on a different index denoted by $q$ instead of $p$. We now want {to} show that these indices are equal and that consequently $V_0=V_\infty$.}

\begin{prop}\label{prop:V=V_0}
  $V_\infty=V_0=V$.
\end{prop}

\begin{proof}

Recall the definitions of \moblue{the} projections $\pi_0,\pi_0',\pi_\infty,$ and $\pi_\infty'$. Then $\mathrm{im}\,\pi_0=V_0(-1), \mathrm{im}\,\pi_0'=V_0(0), \ker\pi_0'=V_0(1), \mathrm{im}\,\pi_\infty=V_\infty(1), \mathrm{im}\,\pi_\infty'=V_\infty(0), $ and $\ker\pi_\infty'=V_\infty(-1).$
Note that $\ker\pi_0'\subseteq V$ and that $\pi_\infty$ is injective on $\ker\pi_0'$. Consequently, $\dim\ker\pi_0'\le \dim\mathrm{im}\,\pi_\infty$, or equivalently $\dim V_0(1)\le \dim V_\infty(1)$.
It was shown before the proposition that
{$$\dim V_0(1)=n-p-1\quad \mathrm{and}\quad \dim V_{\infty}(1)=n-q-1,$$
so $p\ge q$. We want to show that the equality holds.

Write elements of $V$, $V_0$ and $V_\infty$ with respect to block partition of respective sizes $1,k-2,q,p-q$ and $n-k-p+1$ (where some of these numbers may be zero). Then, sets $V_0$ and $V_\infty$ consist of matrices of the form
 \[
    \begin{pmatrix}
   \alpha & b'^{\mathrm{T}} & 0 & 0 & b'''^{\mathrm{T}} \\
   c' & D & 0 & 0 & E \\
   c_1'' & B_1 & \lambda I+A_1' & A_2' & C_1 \\
   c_2'' & B_2 & 0 & \lambda I+A_3' & C_2 \\
   0 & 0 & 0 & 0 & \lambda I+F'
 \end{pmatrix}\ \mbox{respectively}\ \begin{pmatrix}
   \alpha & b'^{\mathrm{T}} & 0 & b_1'''^{\mathrm{T}} & b_2'''^{\mathrm{T}} \\
   c' & D & 0 & E_1 & E_2 \\
   c'' & B & \lambda I+A' & C_1 & C_2 \\
   0 & 0 & 0 & \lambda I+F_1' & F_2' \\
   0 & 0 & 0 & 0 & \lambda I+F_3'
 \end{pmatrix},
 \]
 where $A',A_1',A_3',F',F_1',$ and $F_3'$ are strictly upper triangular. So, members of $V_0(0)\cap V_\infty(0)$ are of the form
 \begin{equation}\label{eq:presek}
    \begin{pmatrix}
   \alpha & 0 & 0 & 0 & 0 \\
   0 & D & 0 & 0 & E \\
   0 & B & \lambda I+A_1' & A_2' & C \\
   0 & 0 & 0 & \lambda I+F_1' & F_2' \\
   0 & 0 & 0 & 0 & \lambda I+F_3'
 \end{pmatrix},
 \end{equation}
 where $A_1',F_1',$ and $F_3'$ are strictly upper triangular.
 An easy computation reveals that
\begin{equation}\label{eq:presek}
\dim(V_0(0)\cap V_\infty(0))=\displaystyle {n-1\choose 2} + {k-1\choose 2}+2 -(p-q)(k-2).
\end{equation}}

Let $\varphi:\ker\pi_0\rightarrow\mathbb{C}^{p-q}$ be the projection defined by $\begin{pmatrix}
                                                                                   \alpha & b^T \\
                                                                                   0 & D
                                                                                 \end{pmatrix}\mapsto b_1'''$, where $b^T=\begin{pmatrix}
                                                                                        b'^T & b''^T & b_1'''^T & b_2'''^T
                                                                                      \end{pmatrix}$;
the blocks of the first matrix are of sizes $1$ and $n-1$ and the blocks of $b^T$ are of the sizes $k-2,q,p-q,$ and $n-p-k+1$. It is obvious that $\ker\pi_0'= V_0(1)\subseteq\ker\varphi$. So, $\varphi$ induces a linear map $\overline{\varphi}: V_0(0)=\mathrm{im}\,\pi_0'\cong \ker\pi_0/ \ker\pi_0'\rightarrow\mathbb{C}^{p-q}$. Let $A=\begin{pmatrix}
                                                 \alpha & 0 \\
                                                 0 & D
                                               \end{pmatrix}\in V_0(0)=\mathrm{im}\,\pi_0'$ be 
arbitrary. Then there exists $b$ such that $A'=\begin{pmatrix}
                                                 \alpha & b^T \\
                                                 0 & D
                                               \end{pmatrix}\in \ker\pi_0\subseteq V$. 
Write $b^T=\begin{pmatrix}
                                                                                        b'^T & b''^T & b_1'''^T & b_2'''^T
                                                                                      \end{pmatrix}$.
Then $b_1'''=\varphi(A')=\overline{\varphi}(A)$. Recall that $\pi_\infty$ is a projection from $V$ to the upper-right corner. The structure of $V_\infty(1)=\mathrm{im}\,\pi_\infty$ implies that all (1,3)\moblue{-}blocks of matrices from $V$ are zero. In particular, $b''=0$. 
If we subtract $A'$ from the matrix
\begin{equation}\label{pazzis}
  A+\begin{pmatrix}
      0 & 0 & 0 & \overline{\varphi}(A)^T & 0 \\
      0 & 0 & 0 & 0 & 0 \\
      0 & 0 & 0 & 0 & 0 \\
      0 & 0 & 0 & 0 & 0 \\
      0 & 0 & 0 & 0 & 0
    \end{pmatrix},
\end{equation}
the difference lies in $V_0(1)=\ker\pi_0'\subseteq V$. It follows that for each $A\in V_0(0)$ the matrix \eqref{pazzis} lies in $V$. In particular, $\ker\overline{\varphi}\subseteq V$. If $A\in V_0(0)\cap V$ then $A\in\ker\pi_{\infty}$, so $\pi_{\infty}'(A)\in \mathrm{im}\, \pi_{\infty}'=V_{\infty}(0)$. Since $\pi_{\infty}'$ is identity on $V_0(0)\cap V$, it follows that $V_0(0)\cap V\subseteq V_0(0)\cap V_{\infty}(0)$. Consequently, $\ker \overline{\varphi} \subseteq V_0(0)\cap V_{\infty}(0)$ and therefore
$$\dim (V_0(0)\cap V_{\infty}(0))\ge \dim \ker \overline{\varphi}\ge \dim V_0(0)-(p-q)=\displaystyle {n-1\choose 2} + {k-1\choose 2}+2-(p-q).$$
Combining this inequality with \eqref{eq:presek} we get $p-q\ge(k-2)(p-q)$.

{If $k\ge 4$, it now immediately follows that}
{$p=q$ and hence $V_0=V_{\infty}$}. 
In particular, $\mathrm{im}\, \pi_{\infty}=\ker \pi_0'$ and $\mathrm{im}\, \pi_0=\ker \pi_{\infty}'$, which implies that $V$ contains all upper-right and lower-left corners of its elements with respect to block partition $(1,n-1)$. Therefore $V=V_0=V_{\infty}$.


 It remains to { get a contradiction in} the case $k=3$ { when $p>q$}. In this case all the above dimension inequalities become equalities. This means that $\overline{\varphi}$ is surjective and its kernel is $V_0(0)\cap V_{\infty}(0)$, so the induced map $\overline{\overline{\varphi}}:\mathbb{C}^{p-q}\cong V_0(0)/(V_0(0)\cap V_{\infty}(0))\to \mathbb{C}^{p-q}$ is \moblue{an} isomorphism. It follows that for each $y\in \mathbb{C}^{p-q}$ the matrix
$$\begin{pmatrix}
0 & 0 & 0 & \overline{\overline{\varphi}}(y)^T & 0\\
0 & 0 & 0 & 0 & 0\\
0 & 0 & 0 & 0 & 0\\
0 & y & 0 & 0 & 0\\
0 & 0 & 0 & 0 & 0
\end{pmatrix}$$
belongs to $V$.

To obtain a contradiction we now 
 adjust the ideas of Lemmas \ref{le:dimen} and \ref{le:22}. 
Assume first that $p-q\ge 2$. Fix distinct nonzero numbers $\lambda_1$ and $\lambda_2$. In the above observation take $y=e_{p-q}$ and write
$\overline{\overline{\varphi}}(e_{p-q})^T=x^T=\begin{pmatrix}
x_1 & x_2 & \cdots & x_{p-q}
\end{pmatrix}$. Note that $V_0(-1)\cap V_{\infty}(-1)\subseteq \ker \pi_{\infty}\subseteq V$ and $V_0(0)\cap V_{\infty}(0)=\ker \overline{\varphi}\subseteq V$. It follows that for an arbitrary $i\in \{1,\ldots ,p-q-1\}$ 
 the matrix
$$\begin{pmatrix}
\lambda_1 & 0 & 0 & x^T & 0 \\
\mu & \lambda_2 & 0 & 0 & 0 \\
0 & 0 & 0 & 0 & 0 \\
0 & e_{p-q} & 0 & E_{i,p-q} & 0 \\
0 & 0 & 0 & 0 & 0
\end{pmatrix}$$ belongs to $V$ for each $\mu \in \mathbb{C}$. Hence it has at most three distinct eigenvalues. Its characteristic polynomial is equal to
\begin{align*}
  \Delta{(t)}&=
(-t)^{n+q-p-2}\begin{vmatrix}
                    \lambda_1-t & 0 & x^{T} \\
                    \mu & \lambda_2-t & 0 \\
                    0 & e_{p-q} & {E_{i,p-q}}-tI
                  \end{vmatrix}\\
&=(-t)^{n-2}(\lambda_1-t)(\lambda_2-t)-{\mu}(-t)^{n+q-p-2}\begin{vmatrix}
                                                                       0 & x^T \\
                                                                       e_{p-q} & {E_{i,p-q}}-tI
                                                                     \end{vmatrix}\\
&={(-t)^{n-4}\Big(t^2(\lambda_1-t)(\lambda_2-t)-\mu(x_i+tx_{p-q})\Big)}.
\end{align*}
 By the assumption on $V$ the quartic polynomial in the parentheses has a multiple zero for each $\mu \in \mathbb{C}$, so for each $\mu \in \mathbb{C}$ its discriminant
$$\begin{vmatrix}
1 & -(\lambda_1+\lambda_2) & \lambda_1\lambda_2 & -\mu x_{p-q} & -\mu x_i & 0 & 0 \\
0 & 1 & -(\lambda_1+\lambda_2) & \lambda_1\lambda_2 & -\mu x_{p-q} & -\mu x_i & 0 \\
0 & 0 & 1 & -(\lambda_1+\lambda_2) & \lambda_1\lambda_2 & -\mu x_{p-q} & -\mu x_i \\
4 & -3(\lambda_1+\lambda_2) & 2\lambda_1\lambda_2 & -\mu x_{p-q} & 0 & 0 & 0 \\
0 & 4 & -3(\lambda_1+\lambda_2) & 2\lambda_1\lambda_2 & -\mu x_{p-q} & 0 & 0 \\
0 & 0 & 4 & -3(\lambda_1+\lambda_2) & 2\lambda_1\lambda_2 & -\mu x_{p-q} & 0 \\
0 & 0 & 0 & 4 & -3(\lambda_1+\lambda_2) & 2\lambda_1\lambda_2 & -\mu x_{p-q} \\
\end{vmatrix}$$
is zero. The above discriminant is a polynomial of degree 4 in $\mu$ with constant term zero.
 The coefficient on $\mu$ is equal to
$${\small \begin{vmatrix}
1 & -(\lambda_1+\lambda_2) & \lambda_1\lambda_2 & 0 & 0 & 0 & 0 \\
0 & 1 & -(\lambda_1+\lambda_2) & \lambda_1\lambda_2 & 0 & 0 & 0 \\
0 & 0 & 1 & -(\lambda_1+\lambda_2) & \lambda_1\lambda_2 & 0 & -x_i \\
4 & -3(\lambda_1+\lambda_2) & 2\lambda_1\lambda_2 & 0 & 0 & 0 & 0 \\
0 & 4 & -3(\lambda_1+\lambda_2) & 2\lambda_1\lambda_2 & 0 & 0 & 0 \\
0 & 0 & 4 & -3(\lambda_1+\lambda_2) & 2\lambda_1\lambda_2 & 0 & 0 \\
0 & 0 & 0 & 4 & -3(\lambda_1+\lambda_2) & 2\lambda_1\lambda_2 & - x_{p-q} \\
\end{vmatrix}}=4\lambda_1^3\lambda_2^3(\lambda_1-\lambda_2)^2x_i$$
and the coefficient on $\mu^4$ is equal to
$${\small\begin{vmatrix}
1 & -(\lambda_1+\lambda_2) & \lambda_1\lambda_2 & -x_{p-q} & -x_i & 0 & 0 \\
0 & 1 & -(\lambda_1+\lambda_2) & 0 & -x_{p-q} & -x_i & 0 \\
0 & 0 & 1 & 0 & 0 & -x_{p-q} & -x_i \\
4 & -3(\lambda_1+\lambda_2) & 2\lambda_1\lambda_2 & -x_{p-q} & 0 & 0 & 0 \\
0 & 4 & -3(\lambda_1+\lambda_2) & 0 & -x_{p-q} & 0 & 0 \\
0 & 0 & 4 & 0 & 0 & -x_{p-q} & 0 \\
0 & 0 & 0 & 0 & 0 & 0 & -x_{p-q} \\
\end{vmatrix}}=-27(x_{p-q})^4.$$
As both coefficients have to be zero, we get $x_{p-q}=x_i=0$. Since $i$ was an arbitrary element from $\{1,\ldots, p-q-1\}$, we get $x=0$, a contradiction with 
the fact that $x=\overline{\overline{\varphi}}(e_{p-q})$ and $\overline{\overline{\varphi}}$ is an isomorphism.

On the other hand, if $p-q=1$, then we consider a similar matrix as above, with the only difference that the $(4,4)$-block is taken to be zero. The characteristic polynomial of this matrix equals
$$\Delta(t)=(-t)^{n-4}\Big(t^2(\lambda_1-t)(\lambda_2-t)-\mu x_{p-q}t\Big).$$
The same argument as above shows that $x_{p-q}=0$, i.e. $x=0$, which yields again a contradiction. This concludes the proof of the proposition.
\end{proof}

Finally we can finish the proof of Theorem \ref{thm:main}. By Proposition \ref{prop:V=V_0} the space $V$ is equal to $V_0$. Consequently, by the argument just before Proposition \ref{prop:V=V_0} there exists $p\in\{0,1,\ldots ,n-k+1\}$ such that with respect to block partition of respective sizes $k-1,p,n-k-p+1$ the space $V$ consists of all matrices of the form $\begin{pmatrix} D & 0 & E \\ B & A+\lambda I & C \\ 0 & 0 & F+\lambda I \end{pmatrix}$, where $A$ and $F$ are strictly upper triangular and all other nonzero blocks are arbitrary. \moblue{A} similarity that exchanges the first two block rows and columns now brings the space $V$ into the form \eqref{eq:required_form}.

\begin{center}
  \textsc{Acknowledgement}
\end{center}

The authors are indebted to an anonymous referee for numerous suggestions which helped us improve substantially the presentation and organization of this manuscript.

\end{document}